\documentclass[conference]{IEEEtran}
\IEEEoverridecommandlockouts

\usepackage{cite}
\usepackage{amsmath,amssymb,amsfonts}
\usepackage{algorithmic}
\usepackage{graphicx}
\usepackage{textcomp}
\usepackage{xcolor}
\usepackage{balance}

\usepackage{bm}
\usepackage{amsthm}

\usepackage{mathtools}
\usepackage{enumitem}
\usepackage{mathrsfs}

\usepackage{amsthm}

\theoremstyle{plain}
\newtheorem{theorem}{Theorem}

\newtheorem{lemma}[theorem]{Lemma}

\theoremstyle{definition}

\newtheorem{assumption}{Assumption}

\newcommand{\cA}{\mathcal{A}}
\newcommand{\cB}{\mathcal{B}}

\newcommand{\cE}{\mathcal{E}}
\newcommand{\cF}{\mathcal{F}}

\newcommand{\cH}{\mathcal{H}}

\newcommand{\cL}{\mathcal{L}}

\newcommand{\cX}{\mathcal{X}}

\newcommand{\bbR}{\mathbb{R}}
\newcommand{\bbN}{\mathbb{N}}

\newcommand{\dd}{\mathrm{d}}

\newcommand{\eps}{\epsilon}
\newcommand{\normal}{\mathsf{N}}

\newcommand{\E}{\mathbb{E}}
\renewcommand{\P}{\mathbb{P}}

\DeclarePairedDelimiter\bkt{[}{]}     
\makeatletter
\newcommand{\@exstar}[1]{\E \bkt*{#1}}
\newcommand{\@exnostar}[2][]{\E \bkt[#1]{#2}}
\newcommand{\ex}{\@ifstar\@exstar\@exnostar}
\makeatother

\makeatletter
\newcommand{\@prstar}[1]{\P \bkt*{#1}}
\newcommand{\@prnostar}[2][]{\P \bkt[#1]{#2}}
\newcommand{\pr}{\@ifstar\@prstar\@prnostar}
\makeatother

\DeclareMathOperator{\cov}{\mathsf{Cov}}

\DeclareMathOperator{\var}{\mathsf{Var}}

\def\BibTeX{{\rm B\kern-.05em{\sc i\kern-.025em b}\kern-.08em
    T\kern-.1667em\lower.7ex\hbox{E}\kern-.125emX}}

\usepackage[
    draft,
    bookmarks=false
]{hyperref}

\newcommand{\wenxuan}[1]{\textcolor{red}{Wenxuan:~#1}}

\begin{document}

\title{Free Energy Universality in Tensor Estimation\\ via Generic Chaining
\thanks{This research was supported in part by NSF Grants 2308445 and 2444150.}
}

\author{
	\IEEEauthorblockN{
		Wenxuan Zou\textsuperscript{*} and Galen Reeves\textsuperscript{\textdagger\textdaggerdbl}
	}
	\IEEEauthorblockA{
		Departments of Physics\textsuperscript{*}, 
		Statistical Science\textsuperscript{\textdagger}, 
		and Electrical and Computer Engineering\textsuperscript{\textdaggerdbl},
		Duke University
	}
}

\maketitle

\begin{abstract}
We study high-dimensional inference problems with tensor-structured data and establish conditions under which their free energy can be approximated by that of a Gaussian comparison model. Our framework applies to models with independent observations and mismatch between the data-generating distribution and the statistical model. The results extend prior work beyond matrix settings and accommodate scaling regimes where the model parameters depend on the dimension.

A key technical contribution is the use of generic chaining to control remainder terms arising from likelihood expansions over tensor-structured parameter spaces. As an application, we establish free energy universality for binary hypergraph models under the minimal assumption of diverging average degree, showing that their asymptotic behavior coincides with that of a Gaussian tensor model, even under model mismatch.
\end{abstract}


\section{Introduction}

A central goal in high-dimensional statistical inference is to characterize the asymptotic behavior of estimators and posterior distributions, including quantities such as mutual information and mean squared error. While many models admit precise characterizations under Gaussian assumptions, extending these results to non-Gaussian and discrete observation models is more difficult, particularly in the presence of higher-order interactions.

A natural approach is to study the \emph{free energy}, which governs the asymptotic behavior of posterior distributions and, via standard perturbation arguments, yields bounds on mean-square error and related performance metrics; see e.g.,~\cite{reeves:2020,guionnet2025estimating,chen:2026a}.

In matrix-valued models, a substantial body of work has shown that the free energy often exhibits a \emph{universality} property: complex observation models can be replaced by Gaussian surrogates without affecting asymptotic behavior~\cite{korada:2010,korada:2011,krzakala:2016,deshpande:2017,lesieur:2017b,lelarge:2018,reeves:2019a,mergny:2024a,guionnet2025estimating,guionnet:2025}. However, these results are largely restricted to pairwise interactions, and much less is known for tensor-structured models with higher-order interactions.


\subsection{Overview of Contributions}

In this paper, we develop a general comparison framework for tensor estimation models with independent observations and $p$-th order interactions. Our main result (Theorem~\ref{thm:main_result}) provides conditions under which the free energy of a general (possibly mismatched) model is well approximated by that of a Gaussian comparison problem. This extends prior universality results beyond matrix settings and accommodates dimension-dependent scaling regimes.

A key technical contribution is a new approach for controlling remainder terms arising from likelihood expansions over tensor-structured parameter spaces. We show that these terms can be bounded using generic chaining \cite{talagrand:2005}  with two metrics that separately capture sub-Gaussian and sub-exponential behavior (Theorem~\ref{thm:chaining_bound}). This distinction is essential for obtaining sharp bounds in high-dimensional regimes.





As an application, we establish free energy universality for binary hypergraph models under model mismatch (Theorem~\ref{thm:sbm}). Under the minimal assumption of diverging average degree, we show that the free energy coincides with that of a corresponding Gaussian tensor model. We further demonstrate that replacing generic chaining with classical chaining, based on a single metric, leads to weaker bounds and restricts the comparison to denser regimes. This highlights the necessity of handling mixed-tail behavior at the level of the parameter space geometry in order to obtain sharp universality results.

\subsection{Related work} 

Universality in high-dimensional inference has been extensively studied using extensions of the Lindeberg principle~\cite{chatterjee:2006,korada:2011}, particularly in matrix and network models~\cite{krzakala:2016,deshpande:2017,lesieur:2017b,lelarge:2018,reeves:2019a,mergny:2024a,guionnet2025estimating,guionnet:2025}. For tensor estimation problems, the limiting behavior under Gaussian models has been established at various levels of generality ~\cite{lesieur:2017b,barbier:2017_tensor_MI, benarous:2020,chen:2022_finite-rank,rossetti:2025,rossetti:2025a}. A version of the universality results developed in this paper was conjectured in~\cite{lesieur:2017b}; we resolve this conjecture and extend it to dimension-dependent models.

Universality phenomena have also been widely studied in statistical physics, particularly for spin glass models~\cite{chatterjee:2006,carmona:2006,talagrand:2011,panchenko:2013}, where one typically considers a fixed Hamiltonian and shows that Gaussian disorder can be replaced by independent non-Gaussian components without affecting the limiting free energy. In contrast, our setting allows both the data distribution and the statistical model to vary with the dimension, requiring comparisons that account for changes in the Hamiltonian induced by mismatch and scaling.

Complementary notions of universality have also been developed for specific algorithms, most notably approximate message passing~\cite{bayati:2012a,lesieur:2015b,lesieur:2017b,chen2021universality,dudeja:2023,wang2024universality}.
In some cases, these results imply free energy universality, but in others, computational-to-statistical gaps prevent algorithmic universality from extending to the level of the free energy.



\section{Problem Formulation} 

\subsection{Tensor Estimation with Mismatched Model}
We study tensor estimation under model mismatch. 
For each problem size $n \in \bbN$, let the parameter space be $\Theta_n = [-1,1]^n$. For $\theta \in \Theta_n$, define the $p$-th order tensor $\eta(\theta) \in \bbR^{n^p}$ by 
\[
\eta_\alpha(\theta)=n^{\frac{1-p}{2}}\,  \theta_{\alpha_1} \cdots \theta_{\alpha_p}, \quad \alpha \in [n]^p
\] 
This normalization ensures that the effective signal-to-noise ratio is order one as $n \to \infty$.

\smallskip
\noindent 
\textbf{Data Distribution:}  For each $\theta \in \Theta_n$, the observations $\{X_\alpha\}_{\alpha \in [n]^p}$ are independent and satisfy
\begin{align*}
X_\alpha \overset{\mathrm{ind}}{\sim} p_n(\cdot \mid \eta_\alpha(\theta))
\end{align*} 
where $p_n(\cdot \mid t), t \in \bbR$ is a family of densities with respect to common base measure on sample space $\cX$.

\smallskip
\noindent \textbf{Statistical Model:} Inference is performed with respect to a prior distortion $\pi_n$ on $\Theta_n$ and postulated log-likelihood of the form 
\begin{align}
L_n(\theta, x) = \sum_{\alpha \in [n]^p} f_n(\eta_\alpha(\theta), x_\alpha) \label{eq:Ln}
\end{align}
where $f_n \colon \bbR \times \bbR \to \bbR$ is a suitably regular function.  The corresponding free energy is defined as
\begin{align*}
F_n(\theta)  \coloneqq \frac{1}{n} \E_{\theta}\left[ \log \int e^{L_n(X, \theta)} \pi_n(\dd \theta)  \right].
\end{align*}
where the expectation is taken under the data distribution. 

\smallskip
\noindent \textbf{Matched Setting:} 
The data distribution and statistical model are said to be \emph{matched} if the postulated log-likelihood corresponds to the log-likelihood ratio of the data distribution relative to $\theta = 0$, i.e.,
\begin{align*}
f_n(t,x) = f_n^*(t,x) \coloneqq \log \frac{p_n(x \mid t)}{p_n(x \mid 0)}.
\end{align*}


\subsection{Gaussian Comparison}

We compare the above model with a Gaussian tensor estimation problem.

\smallskip
\noindent 
\textbf{Comparison Data Distribution:} For each $\theta \in \Theta$, the observations  $\{Y_\alpha\}_{\alpha \in [n]^p}$ are independent and satisfy
\begin{align*}
Y_\alpha \overset{\mathrm{ind}}{\sim} \normal( a_*\,  \eta_\alpha (\theta), \sigma_*^2  ) 
\end{align*}
for parameters $(a_*, \sigma_*) \in \bbR \times (0,\infty)$.

\smallskip
\noindent \textbf{Comparison Statistical Model:} 
Inference is performed with respect to the same prior $\pi_n$ and a log-likelihood of the form
\begin{align*}
L^G_n(\theta, y) &=\!\! \sum_{\alpha \in [n]^p} \! g(\eta_\alpha(\theta),y_\alpha) , \quad g(t,y) = \frac{a yt - \frac{1}{2} a^2t^2}{\sigma^{2}}.
\end{align*}
for parameters $(a,\sigma) \in \bbR \times (0,\infty)$. The corresponding free energy is
\begin{align*}
F^G_n(\theta)  \coloneqq \frac{1}{n} \E_{\theta}\left[ \log \int e^{L^G_n(\theta,Y)} \pi_n(\dd \theta)  \right].
\end{align*}

\smallskip
\noindent \textbf{Matched Setting:} The comparison data distribution and comparison model are matched when $(a_*,\sigma_*) = (a, \sigma)$.

\section{Main Results} 

We consider the high-dimensional setting where $n \to \infty$. Our first result provides a general comparison theorem.

\begin{assumption}\label{assump:free_energy_univer}
For a given tuple $(a_*, \sigma_*, a,\sigma)$, there exists a sequence of statistics $S_n \colon \cX \to \bbR$ and a sequence $\delta_n = o_n(1)$ such that the following hold uniformly for all $\theta \in \Theta_n$:
\begin{enumerate}[label=1.\arabic*,leftmargin=*,]

\item  \textbf{Remainder condition.} \label{ass:remainder} The remainder $R_n \colon \bbR \times \cX \to \bbR$ defined by $R_n(t,x) \coloneqq f_n(t,x)- g(t,S_n(x))$ satisfies 
\begin{align*} 
\E_\theta\bigg[  \sup_{t \in T_n}  \Big| \sum_{\alpha \in [n]^p } R_n(t_\alpha, X_\alpha)\Big |\bigg] \lesssim  n\, \delta_n
\end{align*}
where $T_n \coloneqq \{ \eta(\theta) \mid \theta \in\Theta_n\}$.

\item \textbf{Moment matching.}\label{ass:moment} For all $\alpha \in [n]^p$, the first three moments of $S_n(X_\alpha)$ under the data distribution satisfy \vspace{-.1in}
\begin{align*}
   \big| \E_\theta[S_n(X_\alpha)] - a_* \, \eta_\alpha(\theta) \big|  & \lesssim n^{\frac{1-p}{2}} \delta_n \\
   \big| \var_\theta[S_n(X_\alpha)] - \sigma_*^2  \big|  & \lesssim \delta_n \\
   \E_\theta\Big[ \big| S_n(X_\alpha)- \E_\theta[S_n(X_\alpha)] \big|^3\Big]   & \lesssim n^{\frac{p-1}{2}} \delta_n 
\end{align*}
\end{enumerate}

 \end{assumption}

 \begin{theorem}\label{thm:main_result}
Under Assumption~1 we have  
\begin{align*}
\sup_{\theta \in \Theta_n} |F_n(\theta) - F_n^G(\theta) | \lesssim \delta_n 
\end{align*}
\end{theorem}

Theorem~\ref{thm:main_result} provides general conditions under which the limiting behavior of a tensor estimation problem coincides with that of a Gaussian model. In particular, results established under Gaussian assumptions carry over whenever the conditions of the theorem are satisfied.


While the moment matching condition is relatively standard (arising from the generalized Lindeberg approach) the condition on the remainder term poses significant challenges, especially in the tensor setting ($p > 2$). One of our main contributions is to control this term using generic chaining, as described in the next section.

\subsection{Controlling the Remainder via Generic Chaining}\label{sec:chaining}

This section provides sufficient conditions under which Assumption~\ref{ass:remainder} holds. To simplify notation we suppress the dependence on $n$. For each $\theta \in \Theta_n$, we decompose the remainder as
\begin{align}\label{eq:R_decomposition}
R(t_\alpha, X_\alpha) = \bar{R}_\alpha(t) + Z_\alpha(t),
\end{align}
where for each $\alpha$,  $\{\bar{R}_\alpha(t)\}_{t \in T_n}$ is a centering process that is measurable with respect to $X_\alpha$, and $\{Z_\alpha(t)\}_{t \in T_n}$ is a mean-zero process satisfying $
\E_\theta[ Z_\alpha(t) ] = 0$ for all $t \in T_n$. 
A canonical choice is $\bar{R}_{\alpha}(t) = \E_\theta[ R(t_\alpha, X_\alpha)]$, though other centering schemes may be preferable depending on the application. 

Let $(c_n,v_n) \in (0,\infty)^2$ and define two metrics $d_1$ and $d_2$ on $T_n$ according to  
\begin{align}\label{eq:two_metrics}
d_1(t,u) &= \frac{c_n}{\sqrt{s_n}}\|t-u\|_\infty, ~~d_2(t, u) = \sqrt{\frac{v_n}{s_n}}  \|t- u\|_2,
\end{align}
where $\|\cdot\|_q$ denotes the entry-wise $\ell_q$ norm and $s_n=n^{p-1}$.



\begin{assumption}\label{cond:bernstein} 
For each $n \in \bbN$, there exist $(v_n,c_n,\delta_n)$ such that the following holds uniformly for all $\theta \in \Theta_n$:
\begin{enumerate}[label=2.\arabic*,leftmargin=*,]

\item \textbf{Centering condition.}\label{cond:2_centering} 
    \begin{align*}
  \E_\theta\bigg[  \Big| \sup_{t \in T_n} \sum_{\alpha \in [n]^p }  \bar{R}_\alpha(t) \Big|\bigg] \lesssim  n\, \delta_n
 \end{align*}
 \item \textbf{Regularity condition.}\label{cond:2_integ} $t \mapsto \sum_{\alpha \in [n]^p }  Z_\alpha(t) $ is continuous a.s., and there exists $t_0 \in T_n$ such that
     \begin{align*}
  \E_\theta\bigg[  \Big| \sum_{\alpha \in [n]^p }  Z_\alpha(t_0) \Big|\bigg] \lesssim  n\, \delta_n
 \end{align*}
\item \textbf{Mixed-tail increment condition.}\label{cond:2_tails}
\begin{align*}
	\sum_{\alpha\in[n]^p}\!\!\E_\theta \left[ \big | Z_\alpha(t)  - Z_\alpha(u) \big|^q\right] \le \frac{q!}{2} d_{2}^2(t,u)\,   d^{q-2}_{1}(t,u)
\end{align*}
for all $t,u \in T_n$ and integers $q\ge 2$.
\end{enumerate}

\end{assumption}


\begin{theorem}\label{thm:chaining_bound}
Under Assumption~\ref{cond:bernstein}, we have  
    \begin{align*}
    \sup_{\theta \in \Theta_n}   \E_\theta\bigg[  \sup_{t \in T_n}  \Big| \sum_{\alpha \in [n]^p } & R_n(t_\alpha,  X_\alpha)\Big |\bigg]\\ 
     & \lesssim c_n\,n^{2-p}\, p 
    + \sqrt{v_n}\, n^{\frac{3-p}{2}}\,p + n\delta_n.
    \end{align*}
\end{theorem}

By Theorem~\ref{thm:chaining_bound}, Assumption~\ref{ass:remainder} holds under Assumption~\ref{cond:bernstein} and the conditions $c_n \lesssim n^{p-1}\delta_n$ and $v_n \lesssim n^{p-1}\delta_n^2$, provided that the sequences $\delta_n$ in the two assumptions coincide.

\subsection{Application to Binary Hypergraph Model}\label{sec:hypergraph}

We specialize our results to a binary hypergraph model with independent observations
\[
X_\alpha \overset{\mathrm{ind}}{\sim} \mathrm{Bernoulli}\bigg(
\frac{d_n}{s_n} + \sqrt{\frac{d_n}{s_n}\Big(1-\frac{d_n}{s_n}\Big)}\,\sqrt{\lambda_*}\,\eta_\alpha(\theta)
\bigg),
\]
where $s_n = n^{p-1}$ is a normalization factor ensuring a non-degenerate high-dimensional limit, $d_n$ controls (to leading order) the average degree,  
and $\lambda_* \in (0, \infty)$ corresponds to the Fisher information at $\theta = 0$. 

For the data distribution, the log-likelihood ratio relative to $\theta = 0$ takes the form \eqref{eq:Ln}, with component-wise contribution
\begin{align*}
f_n^*(t,x) = x\log(1 + b_n \sqrt{\lambda_*} t) +  (1\!-\!x)\log\!\Big(1 -\!  \frac{\sqrt{\lambda_*}t}{b_n}  \Big) 
\end{align*}
where $b_n  =\sqrt{ s_n /d_n -1}$.

We consider a (possibly mismatched) statistical model with log-likelihood of the form \eqref{eq:Ln}, with
\begin{align}\label{eq:model_f}
f_n(t,x) = x\log(1 + b_n \sqrt{\lambda} t) +  (1\!-\!x)\log\!\Big(1 -\!  \frac{\sqrt{\lambda}t}{b_n}  \Big) 
\end{align}
where $\lambda \in (0, \infty)$. Under this formulation, the data distribution and statistical model are matched when $\lambda = \lambda_*$. 


\begin{theorem}\label{thm:sbm}
Given $(\lambda, \lambda_*) \in (0,\infty)^2$ let the parameters of the comparison be given by 
\begin{alignat*}{3}
a_* &= \sqrt{\lambda_*}, \quad   \sigma_* = 1, \quad a= \sqrt{\lambda},  \quad  \sigma = 1.
\end{alignat*}
Assume the average degree satisfies  $d_n \to \infty$ and $s_n -d_n \to \infty$. Then, 
    \begin{align*}
      \sup_{\theta \in\Theta_n}  |F_n(\theta) - F_n^G(\theta)| \lesssim \frac{1}{\sqrt{d_n}}+\frac{1}{\sqrt{s_n-d_n}}.
    \end{align*}
\end{theorem}

A natural question is whether simpler methods, such as Bernstein-type inequalities combined with classical chaining (e.g.,~\cite[Lemma~13.1]{boucheron:2013}), would suffice to control the remainder terms. These approaches rely on a single metric to control both the sub-Gaussian and sub-exponential components in the remainder process. In the context of Theorem~\ref{thm:chaining_bound}, this restriction introduces an additional factor of $n^{p/2}$ in the term involving $c_n$, and consequently imposes significantly stronger requirements on the model parameters. In the hypergraph setting, $c_n$ grows as the graph becomes sparse, and bounds obtained via classical chaining would therefore require the average degree to scale at a prescribed rate in order to keep the bound sublinear.

By contrast, the generic chaining approach used in this paper, specifically the mixed-tail bounds in~\cite[Theorem 2.2.23]{talagrand:2014}, fully exploits the mixed-tail structure by handling $d_1$ and $d_2$ separately. This allows us to obtain the correct order and establish universality under the minimal assumption that the average degree diverges, without any restriction on its rate.

\section{Proof of Main Results}

\subsection{Proof of Theorem~\ref{thm:main_result}}
Introduce an intermediate log-likelihood and the corresponding free energy,
\begin{align*}
\tilde{L}_n(\theta,x) &\coloneqq    \sum_{\alpha\in[n]^p}\,g(\eta_\alpha(\theta),S_n(x_\alpha)),\\
\tilde{F}_n(\theta) &\coloneqq \frac{1}{n} \E_{\theta}\left[ \log \int e^{\tilde{L}_n(\theta,x)} \pi_n(\dd \theta)  \right].
\end{align*}

\subsubsection{Bound $|F_n(\theta)-\tilde{F}_n(\theta)|$ via Assumption~\ref{ass:remainder}}
By the Lipschitz continuity of the log-partition (log-sum-exp) functional and Jensen's inequality,
\begin{align*}
    |F_n(\theta)-\tilde{F}_n(\theta)|&\le \frac{1}{n}\E_\theta\Big[\,\sup_{\theta\in\Theta_n}|L_n(\theta,X)-\tilde{L}_n(\theta,X) |\, \Big]\\
    &\le\frac{1}{n}\E_\theta\bigg[  \sup_{t \in T_n}  \bigg| \sum_{\alpha \in [n]^p } R(t_\alpha, X_\alpha)\Big |\bigg].
\end{align*}
By Assumption~\ref{ass:remainder},
\begin{align}\label{eq:first_compare}
   \sup_{\theta\in\Theta_n} |F_n(\theta)-\tilde{F}_n(\theta)| \lesssim \delta_n.
\end{align}

\subsubsection{Bound $|F_n^G(\theta)-\tilde{F}_n(\theta)|$ via Assumption~\ref{ass:moment}}

\begin{lemma}\label{lemma:stein}
    Let $X$ and $Y$ be real-valued random variables, with $Y$ Gaussian. Let $\eta \sim \nu$, where $\nu$ is supported on a compact subset of $\bbR$ and $|\eta| \le b$ almost surely. Define $\phi(x) = \log \int \exp(x\eta)\nu(\dd \eta)$ and let $\Delta\coloneqq |\E[\phi(X)]-\E[\phi(Y)]|$,  \begin{equation}\label{eq:delta_stein}
        \begin{aligned}
        \Delta\le &|\E[X]-\E[Y]|b\\
        &+|\var(X)-\var(Y)|b^2 +4\E|X-\E[X]|^3b^3.
    \end{aligned}
    \end{equation}
\end{lemma}

\begin{proof}
Denote by $\tilde{X}\coloneqq X-\E[X]$, $\tilde{Y}\coloneqq Y-\E[Y]$. 
Define $\|h\|_\infty\coloneqq\sup_{x\in\bbR}|h(x)|$ for real-valued function $h$. Define an interpolation $W_s = s \E[X] + (1-s)\E[Y] + \sqrt{s}\tilde{X} + \sqrt{1-s}\tilde{Y}$ for $s\in[0,1]$, then
\begin{align*}
    \E[\phi(X)] - \E[\phi(Y)] &= \int_0^1 \!\E[\phi'(W_s)]\big[\E[X]-\E[Y]\big]\dd s\\
    &+\frac{1}{2}\int_0^1 \!\E\Big[\frac{\phi'(W_s)}{\sqrt{s}}\tilde{X}\!-\!\frac{\phi'(W_s)}{\sqrt{1\!-\!s}}\tilde{Y}\Big]\dd s.
\end{align*}
By Gaussian integration by parts and 
Lemma~3 in \cite{carmona:2006},
\begin{align*}
\E[\phi'(W_s)\tilde{Y}] &= \sqrt{1-s}\,\E[\phi''(W_s)]\var(Y),\\
\E[\phi'(W_s)\tilde{X}] &\le \sqrt{s}\,\E[\phi''(W_s)]\var(X)+\frac{3s}{2}\cE,
\end{align*}
where $\cE = \|\phi'''\|_\infty \E|\tilde{X}|^3$. Then
\begin{align*}
    \E[\phi(X)] - \E[\phi(Y)] &\le  \|\phi'\|_\infty|\E[X]-\E[Y]| \\
    &+ \|\phi'' \|_\infty|\var(X)-\var(Y) |+\frac12\cE.
\end{align*}
Finally, note that $\phi$ is a cumulant generating function of $\eta$,
\begin{align*}
    \|\phi'\|_\infty \le b,~~\|\phi''\|_\infty \le 2\,b^2,~~\|\phi'''\|_\infty \le 8\,b^3.
\end{align*}
Using these bounds, we show one side of the inequality~\eqref{eq:delta_stein}. A symmetric argument shows the other side.
\end{proof}

Define $\cL_n$ as the class of likelihood functions
$L_n\colon\Theta_n \times \bbR^{n^p} \to \bbR$ with $\int \exp\big(L_n(\theta,x)\big)\pi_n(\dd\theta) < \infty$. Then, define the functional $\Phi_n\colon\cL_n \to \{\bbR^{n^p} \to \bbR\}$ by
\begin{align*}
\Phi_n[L_n](x)
\coloneqq
\log \int \exp\big(L_n(\theta,x)\big)\,\pi_n(\dd\theta).
\end{align*}
For notational simplicity, we write $\Phi_n[\,\cdot\,]$ for $\Phi_n[\,\cdot\,](x)$.

Denote $g_{\alpha}^{(1)}\coloneqq g(\eta_\alpha(\theta),S_n(X_\alpha))$ and $g_{\alpha}^{(2)}\coloneqq g(\eta_\alpha(\theta),Y_\alpha)$. Let \(\{\alpha^{(1)},\dots,\alpha^{(n^p)}\}\) be an enumeration of the index set. Then, for $k=0,\dots,n^p$, define\vspace{-0.2em}
\begin{align*}
\tilde{g}^{(k)}\coloneqq \sum_{1\le j\le k} g^{(1)}_{\alpha^{(j)}} + \sum_{k<j\le n^p} g^{(2)}_{\alpha^{(j)}}.
\end{align*}
Following a telescoping argument and triangle inequality,\vspace{-0.2em}
\begin{align*}
    |\tilde F_n(\theta) - F_n^G(\theta)| &=\frac{1}{n}\big| \E_\theta[\Phi_n(\tilde{g}^{(n^p)})] -\E_\theta[\Phi_n(\tilde{g}^{(0)})] \big|   \\
    &\le \frac1n \sum_{k=1}^{n^p}\big|\E_\theta\big[
\Phi_n[\tilde{g}^{(k)}] - \Phi_n(\tilde{g}^{(k-1)})
\big]\big|.
\end{align*}
Denote $\Delta_{n,k} \coloneqq 
\Phi_n[\tilde{g}^{(k)}] - \Phi_n[\tilde{g}^{(k-1)}]$.
For each $k\in[n^p]$, by Lemma~\ref{lemma:stein} with $(X,Y,\eta)=(S_n(X_{\alpha^{(k)}}), Y_{\alpha^{(k)}},\frac{a}{\sigma^2}\,\eta_{\alpha^{(k)}}(\theta))$, and Jensen's inequality, we have
\begin{align*}
|\E_\theta\Delta_{n,k}|
&\le \E_\theta \big|\E_\theta\big[\Delta_{n,k}\mid \{(S_n(X_{\alpha^{(j)}}), Y_{\alpha^{(j)}})\}_{j=1,j \neq k}^{n^p}\big]\big|\\
&\le \big|\E_\theta[S_n(X_{\alpha^{(k)}})]-\E_\theta[Y_{\alpha^{(k)}}]\big|\,\,n^{(1-p)/2}\\
&\quad +\big|\var_\theta(S_n(X_{\alpha^{(k)}}))-\var_\theta(Y_{\alpha^{(k)}})\big|\,\,n^{1-p}\\
&\quad +\E_\theta|S_n(X_{\alpha^{(k)}}) - \E_\theta[S_n(X_{\alpha^{(k)}})]|^3\,\,  n^{3(1-p)/2}.
\end{align*}
By Assumption~\ref{ass:moment}, $|\E_\theta\Delta_{n,k}| \le {n}^{1-p}\delta_n$ for all $k\in[n^p]$. So,
\begin{align}\label{eq:second_compare}
    \sup_{\theta\in\Theta_n}|F_n^G(\theta)-\tilde{F}_n(\theta)|\lesssim\delta_n.
\end{align}
By the triangle inequality, \eqref{eq:first_compare} and~\eqref{eq:second_compare} together complete the proof of Theorem~\ref{thm:main_result}.

\subsection{Proof of Theorem~\ref{thm:chaining_bound}}\label{sec:proof_chaining} 
The proof of the theorem hinges on showing that the Assumption~\ref{cond:2_tails} implies~\eqref{eq:chaining_bound_proof} using generic chaining. The result then follows by the triangle inequality from~\eqref{eq:chaining_bound_proof} and Assumption~\ref{cond:2_centering}--\ref{cond:2_integ}.

We introduce some concepts in generic chaining~\cite{talagrand:2005}. A pseudometric space $(T,d)$ consists of a set $T$ and a non-negative function $d \colon T \times T \to [0, \infty)$ that is a pseudometric.
\begin{itemize}[leftmargin=9pt]
    \item The diameter is $\Delta(T,d)  \coloneqq \sup_{t,u \in T} d(t,u)$;
    \item The covering number $N(T,d,\eps)$ is the minimal number of radius-$\eps$ balls needed to cover $T$;
    \item A central quantity is the $\gamma_\alpha$-functional as defined in~\cite[Definition 2.2.19]{talagrand:2014}, which admits the upper bound as discussed in~\cite[Section~2]{dirksen:2015},
    \begin{align}\label{eq:gamma_ub}
    	\gamma_\alpha(T,d) \le C_\alpha \int_0^\infty \left( \log N(T,d,u) \right)^{1/\alpha} \dd u,
    \end{align}
    where $C_\alpha$ is a constant depending only on $\alpha$. 
\end{itemize}

\begin{lemma}[Lipschitz mapping, Theorem~1.3.6 (b) in~\cite{talagrand:2005}]\label{lemma:Lip_mapping}
Let $(U,d')$ and $(T,d)$ be pseudometric spaces. Suppose that $\phi \colon U \to T$ is surjective and there exists a constant $A > 0$ such that
\[
d\big(\phi(x),\phi(y)\big) \le A\, d'(x,y), \qquad \forall x,y \in U.
\]
Then, for any $\alpha > 0$, $\gamma_\alpha(T,d) \le C_\alpha\, A\, \gamma_\alpha(U,d')$.
\end{lemma}

\begin{lemma}
\label{lemma:rank_one_T}
Let $
 T = \{ n^{\frac{1-p}{2}} \,\theta^{\otimes p}  \mid \theta \in  [-1,1]^n \}$. Then,  the metrics $d_1$ and $d_2$ defined in~\eqref{eq:two_metrics} satisfy 
 \begin{align*}
    \gamma_1(T,d_1)\lesssim n^{2-p}\,p\,c_n,\quad \gamma_2(T,d_2)\lesssim n^{\frac{3-p}{2}}\,p\,\sqrt{v_n}.
 \end{align*}
\end{lemma}

\begin{proof}
For all $t, u \in T$ there exists $x,y \in [-1,1]^p$ such that $t =n^{\frac{1-p}{2}}  x^{\otimes p}$ and $u = n^{\frac{1-p}{2}} y^{\otimes p} $. The triangle inequality together with the fact that $\|\theta^{\otimes p} \|_q = \|\theta\|_q^p \le n^{p/q}$ for all $\theta \in [-1,1]^n$ gives  
\begin{align*}
\|t - u\|_q &\le n^{\frac{1-p}{2}} \, \Big( \sum_{r=1}^{p} \|x\|_q^{r-1} \|y\|_q^{p-r}  \Big) \| x - y\|_q\\
&\le   n^{\frac{1-p}{2}}  n^{\frac{p-1}{q}} \, p  \, \| x - y\|_q .
\end{align*}
By Lemma~\ref{lemma:Lip_mapping}, the $\gamma_{\alpha}$ functional satisfies 
\begin{align*}
\gamma_1(T,d_1) & \le   n^{1-p} \,c_n\, p \,    \gamma_1\left([-1,1]^n,\|\cdot\|_\infty \right)  \\
\gamma_2(T,d_2) &\le n^{\frac{1-p}{2}} \, \sqrt{v_n}\, p\,  \gamma_2\left([-1,1]^n,\|\cdot\|_2 \right)
\end{align*}
A standard calculation for the covering numbers yields
\begin{align*}
N\left([-1,1]^n , \|\cdot\|_\infty ,  \eps  \right) &\le  \Big(  \frac{2}{\eps}  \wedge 1 \Big)^n , \\
N\left([-1,1]^n , \|\cdot\|_2 ,\eps \right)   &\le \Big(  \frac{2\sqrt{n}  }{\eps}  \wedge 1 \Big)^n.
\end{align*}
By~\eqref{eq:gamma_ub} it follows that
\begin{align*}
	\gamma_1\left([-1,1]^{n} , \|\cdot\|_\infty \right)& \lesssim  2 n  \int_0^1   \log(1/u) \, \dd u  \lesssim  n,  \\
	\gamma_2\left([-1,1]^{n} , \|\cdot\|_2 \right) &\lesssim   2 n  \int_0^{1}  \sqrt{  \log(1/v) } \, \dd v     \lesssim  n.
\end{align*}
Combining these results gives the stated bounds. 
\end{proof}

Under Assumption~\ref{cond:2_tails} it follows from \cite[Theorem~2.10]{boucheron:2013} that $\hat{Z}(t)\coloneqq\sum_{\alpha}Z_\alpha(t)$ satisfies a mixed-tail property: 
\begin{align*}
	\pr*{|\hat{Z}(t) - \hat{Z}(u) | \ge \sqrt{2s} \, d_2(t,u) + s\,  d_1(t,u)}   \le 2 e^{-s}
\end{align*}
for all $t, u \in T_n$ and $s \ge 0$.
By~\cite[Theorem~2.2.23]{talagrand:2014}, it follows that for any finite set $S\subset T_n$ containing $t_0$,
\begin{equation*}
\begin{aligned}
	\E_{\theta}\Big[ \sup_{t\in S} |\hat{Z}(t) - \hat{Z}(t_0) | \Big] &\lesssim \gamma_2(T_n, d_2) + \gamma_1(T_n, d_1)
\end{aligned}
\end{equation*}
By compactness of $T_n$ and a.s.\ continuity of $t \mapsto \hat{Z}(t)$ (Assumption~\ref{cond:2_integ}), this bound extends to the expected supremum over $T_n$, by combining a standard $\varepsilon$-net approximation with monotone convergence. Then, by Lemma~\ref{lemma:rank_one_T}, 
\begin{align}~\label{eq:chaining_bound_proof}
\E_{\theta}\Big[ \sup_{t\in T_n} |\hat{Z}(t) - \hat{Z}(t_0) | \Big] \lesssim  n^{2-p}\,p\,c_n + n^{\frac{3-p}{2}}\,p\,\sqrt{v_n}.
\end{align}
Combining~\eqref{eq:chaining_bound_proof} with Assumption~\ref{cond:2_centering}--\ref{cond:2_integ} proves Theorem~\ref{thm:chaining_bound}.

\subsection{Proof of Theorem~\ref{thm:sbm}}\label{sec:proof_sbm}
In this proof, we specify $\delta_n = 1/\sqrt{d_n}+ 1/\sqrt{s_n-d_n}$.
With this choice, we prove this theorem by verifying Assumptions~\ref{ass:remainder} and~\ref{ass:moment} and then applying Theorem~\ref{thm:main_result}. Verification of Assumption~\ref{ass:remainder} follows from Theorem~\ref{thm:chaining_bound}.

\subsubsection{Verify Assumption~\ref{ass:moment}} Consider the statistic 
\begin{align*}
    S_n(x) = \frac{\sigma^2}{a}\partial_t f_n(0,x) =(b_n\,x+ b_n^{-1}(x-1)).    
\end{align*}
Note that, for any $q>0$, 
\begin{align}\label{eq:bn}
    b_n^q+b_n^{-q}\le (b_n^2+b_n^{-2}+2)^{\frac{q}{2}}\le(s_n\delta_n^2)^{\frac q2}.
\end{align}
By direct calculations with this inequality, we have,
\begin{align*}
    |\E_\theta[S_n(X_\alpha)]-a_*\eta_\alpha(\theta)|&=0 \\
   |\var(S_n(X_\alpha)) - \sigma^2_*| &\lesssim (b_n + b_n^{-1})\, s_n^{-\frac12} \le \delta_n.\\
   \E_\theta[|S_n(X_\alpha)-\E_\theta[S_n(X_\alpha)]|^3]
    &\lesssim (b_n+b_n^{-1})\le s_n^{\frac12}\delta_n,
\end{align*}
which verify Assumption~\ref{ass:moment}.

\smallskip
\subsubsection{Verify Assumption~\ref{ass:remainder}} The verification relies on verifying Assumption~\ref{cond:2_centering},~\ref{cond:2_integ} and~\ref{cond:2_tails} and applying Theorem~\ref{thm:chaining_bound}. Consider the series expansion of~\eqref{eq:model_f} at $t=0$,
\begin{align*}
f_n(t,x) = \sqrt{\lambda}S_n(x)\,t + \frac{1}{2}\partial_t^2f_n(0,x) t^2 + \frac{1}{6}\partial_t^3 f_n(r,x)t^3,
\end{align*}
where $r\in[0,t]$.
By basic identity of the log likelihood,
\begin{align*}
    \lambda\E_0[S_n^2(x)] = -\E_0[\partial_t^2 f_n(0,x)] = \lambda = a^2.
\end{align*}
Therefore, the remainder $R_n(t,x) =f_n(t,x)- g(t,S_n(x))$ has an explicit form,
\begin{align*}
    R_n(t, x) =\frac{1}{2}\Big[\partial_t^2 f_n(0,x) +\lambda  \Big]t^2+\frac{1}{6}\partial_t^3 f_n(r,x)t^3.
\end{align*}
Using the notation of Section~\ref{sec:chaining}, we specify the decomposition~\eqref{eq:R_decomposition} by letting,
\begin{align*}
    \bar{R}_{\alpha}(t)&= \frac{1}{2}\Big[ \E_\theta[\partial_t^2 f_n(0,X_\alpha)]+\lambda\Big]t^2+\frac{1}{6}\partial_t^3 f_n(r,X_\alpha)t^3,\\
    Z_\alpha(t) &= \frac{1}{2}\Big[  \partial_t^2 f_n(0,X_\alpha)-\E_\theta[\partial_t^2 f_n(0,X_\alpha)]\Big]t^2.
\end{align*}
For $\bar{R}_{\alpha}(t)$, direct calculation with \eqref{eq:bn} shows for all $\alpha\in[n]^p$,
\begin{align*}
\E_\theta\big[\sup_{t\in T_n}|\bar{R}_\alpha(t)|\big]\lesssim s_n^{-1}\,\delta_n.
\end{align*}
Sum over $\alpha$ verifies Assumption~\ref{cond:2_centering}. 

\smallskip
\noindent
Then, define $\hat{Z}(t) \coloneqq \sum_{\alpha\in[n^p]} Z_\alpha(t)$, which is
\begin{align*}
    \hat{Z}(t)  = \frac{1}{2}\sum_\alpha \big[ \partial_t^2 f_n(0,X_\alpha)-\E_\theta[\partial_t^2 f_n(0,X_\alpha)]\big] t_\alpha^2.
\end{align*}
With a slight abuse of notation, we set $t_0=0$ in Assumption~\ref{cond:2_integ}. 
Then $\hat{Z}(t_0)=0$, so the Assumption~\ref{cond:2_integ} is verified.

\smallskip
\noindent
Finally, let $V_\alpha\coloneqq \partial_t^2 f_n(0,X_\alpha)-\E_\theta[\partial_t^2 f_n(0,X_\alpha)]$. Then, 
\begin{align*}
    \E_\theta[V_\alpha^2]&= \lambda(b_n^2-b_n^{-2})^2\,\var_\theta[ \,X_\alpha]
    \le 2\,\lambda\, s_n\,\delta_n^2,\\
    |V_\alpha|&\le 2\lambda (b_n^2 +b_n^{-2})\le 2\,\lambda\, s_n\,\delta_n^2.
\end{align*}
Then, for all $q\ge 2$, 
\begin{align*}
\E_\theta[|V_\alpha|^q]\le \E_\theta[|V_\alpha|^2] (2\lambda s_n\delta_n^2)^{q-2}\le (2\,\lambda\, s_n \delta_n^2)^{q-1}.
\end{align*}
Therefore,
$V_\alpha$ with $ v_n=c_n=2\lambda\,s_n\,\delta_n^2$ satisfies 
\begin{align*}
\E_\theta[ |V_\alpha|^q ] \le \frac{q!}{2}v_nc_n^{q-2},~\forall q\ge 2.
\end{align*}
As a consequence, for all $\alpha\in[n]^p$ and all $t,u\in T_n$,
\begin{equation}\label{eq:Z_sub_gamma}
\begin{aligned}
    \E_\theta[ |Z_\alpha(t) - Z_\alpha(u)|^q ] 
    &= \frac{1}{2^q}\E_\theta[|V_\alpha|^q|t_\alpha^2-u_\alpha^2|^q]\\
    &\le \frac{q!}{2} d_{2,\alpha}^2(t,u)\, d_{1,\alpha}^{q-2}(t,u)
\end{aligned}
\end{equation}
where we define,
\begin{align*}
    d_{1,\alpha}(t,u)= \frac{c_n}{\sqrt{s_n}}| t_\alpha - u_\alpha|,\quad d_{2,\alpha}(t,u)=\sqrt{\frac{v_n}{s_n}} | t_\alpha - u_\alpha| .
\end{align*}
Summing over $\alpha$ on both sides of~\eqref{eq:Z_sub_gamma} verifies Assumption~\ref{cond:2_tails}. 
Since Assumptions~\ref{cond:2_centering},~\ref{cond:2_integ}, and~\ref{cond:2_tails} are all satisfied, 
we apply Theorem~\ref{thm:chaining_bound} with $c_n = v_n = 2\lambda\,s_n \delta_n^2$, which verifies Assumption~\ref{ass:remainder}. 
This completes the proof of Theorem~\ref{thm:sbm}.



\section{Conclusion and Future Work}

We established a general universality result for tensor estimation problems, showing that their free energy is well approximated by that of a Gaussian comparison model under broad conditions. Our approach uses generic chaining to control remainder terms, enabling us to handle tensor-structured models and dimension-dependent regimes. As an application, we proved universality for binary hypergraph models under minimal assumptions on the average degree.

Several directions for future work remain. One is to relax the conditions on the remainder process to accommodate heavy-tailed noise. Another is to extend the framework to models with dependent observations, such as those arising from orthogonally invariant (but non-Gaussian) matrix ensembles.

\clearpage
\balance
\bibliographystyle{IEEEtran}
\bibliography{long_names,library,universality,more_refs}

\end{document}